\documentclass[a4paper, 11pt]{amsart}
\input xy
\swapnumbers
\usepackage[top=4cm, bottom=3.5cm, left=2cm, right=2cm,marginparwidth=1.8cm, marginparsep=0.1cm]{geometry}
\usepackage[OT2, T1]{fontenc}
\usepackage[english]{babel}
\usepackage{amssymb}
\usepackage{amsthm}
\usepackage{enumerate}
\usepackage[active]{srcltx}
\usepackage{amsmath}
\usepackage{nicefrac}
\usepackage{tikz-cd} 
\usepackage[colorlinks=true]{hyperref}
\usepackage{lipsum}
 \usepackage{xfrac}
\usepackage{faktor}
\usepackage{scalerel}
\usepackage[inline]{enumitem}

\newtheorem{proposition}{Proposition}[section]
\newtheorem{theorem}[proposition]{Theorem}
\newtheorem{lemma}[proposition]{Lemma}
\newtheorem{corollary}[proposition]{Corollary}

\newtheorem{claim}[proposition]{Claim}

\newtheorem{definition}[proposition]{Definition}

\theoremstyle{remark}
\newtheorem{remark}[proposition]{Remark}

\newtheorem{examples}[proposition]{Examples}

\newtheorem*{remark*}{Remark}
\newtheorem*{remarks*}{Remarks}

\numberwithin{equation}{section}

\newcommand{\R}{\mathbb{R}}
\newcommand{\N}{\mathbb{N}}
\newcommand{\Z}{\mathbb{Z}}

\newcommand{\vW}{\stackrel{w}{\ast}}

\makeatletter
\DeclareRobustCommand\bigop[2][1]{%
 \mathop{\vphantom{\bigoplus}\mathpalette\bigop@{{#1}{#2}}}\slimits@
}
\newcommand{\bigop@}[2]{\bigop@@#1#2}
\newcommand{\bigop@@}[3]{%
 \vcenter{%
 \sbox\z@{$#1\bigoplus$}%
 \hbox{\resizebox{\ifx#1\displaystyle#2\fi\dimexpr\ht\z@+\dp\z@}{!}{$\m@th#3$}}%
 }%
}
\makeatother

\begin{document}
\title[Word problem in verbal products]{The word problem in verbal products of groups}

\author[J. Brude, R. Sasyk]{Javier Brude $^{1}$ \MakeLowercase{and} Rom\'an Sasyk $^{2,3}$}

\address{$^{1}$Villa Luro, Buenos Aires, Argentina.}

\address{$^{2}$Instituto Argentino de Matem\'atica Alberto P. Calder\'on-CONICET,
Saavedra 15, Piso 3 (1083), Buenos Aires, Argentina.}

\address{$^{3}$Departamento de Matem\'atica, Facultad de Ingenier\'ia, Universidad de Buenos Aires, Argentina.}

\email{\textcolor[rgb]{0.00,0.00,0.84}{javierbrd@gmail.com}}
\email{\textcolor[rgb]{0.00,0.00,0.84}{rsasyk@fi.uba.ar}}


\date{}

\begin{abstract}	
We prove that nilpotent and metabelian products of groups with decidable word problem have decidable word problem. In contrast, we exhibit families of groups with decidable word problem for which their solvable product has undecidable word problem. In doing so, we exhibit hyperbolic groups with undecidable membership problem in their derived subgroups.
\end{abstract}

\maketitle
\section{Introduction}

Given a family of groups, the direct sum and the free product provide ways of constructing new groups out of it.
Even though both operations are quite different, they share some common features, like commutativity,  associativity, etc.
In \cite{KUR53}, Kurosh asked if there were other operations on a family of groups alike these two. 
This problem was solved in the affirmative in  \cite{GOL56METAB,GOL56NILP}, where, for each $k\in \N$, 
Golovin defined the $k${\it -nilpotent product} of groups and analysed some structural properties of them. 
In  \cite{MOR56,MOR58}, 
Moran framed nilpotent products as an instance of his more general  construction of  {\it verbal products of groups}, 
a way of defining operations in groups that fulfil Kurosh's requirements.

One notable fact that motivated our early research on this topic
is that verbal products interpolate between the direct product and the free product: each variety of groups determines 
one such product, the two extreme cases being the variety of all groups, which gives the free product, and the abelian variety, which gives the direct
 product.  In \cite{SAS18,MR4502610} we undertook the analysis of
verbal products from the point of view of geometric and measurable group theory. 
In particular, in those articles we analysed several permanence properties that are of interest in these areas,
 among them: amenability, the Haagerup property,  Kazhdan's property (T),  soficity,  hyperlinearity and orderability.

The present paper continues that line of work, but now on the algorithmic side: we ask which verbal
products preserve the decidability of the word problem. 
The question makes sense since the most studied examples of verbal products, namely
the direct product and the free product, obviously preserve it. In stark contrast, we show that for other verbal products
the answer depends on the variety in a drastic way. On the positive side, the nilpotent products 
and the metabelian product always preserve 
word problem decidability, see Theorem \ref{nilpotentcase}. For
 the rest of the solvable products one can find counterexamples, 
see Theorems \ref{main theorem} and \ref{last theorem}, which are at the core of the article.
Along the way we obtain a result which seems to us of independent interest:
Let $0<\lambda\leq 1$. For every $n\geq3$ there exists a finitely presented group with small cancellation property $C'(\lambda)$
 such that there is no algorithm deciding whether a given element belongs to the $n$-th derived subgroup.
 In particular, by taking $\lambda \leq 1/6$, in Proposition \ref{prop princ} we exhibit a hyperbolic group with the same property.

 \section{Preliminaries about verbal products}

In this section we list the few facts about verbal products of groups that will be needed in the article.
The notations, conventions and terminology used here are the ones from \cite{MR4502610}, and we refer the reader to that paper for more on verbal products and verbal wreath products.
In particular, for us $[a,b]= aba^{-1}b^{-1}$  and $A^b=bAb^{-1}$.\\

Let $F_\infty$ be the free group on countably many letters $\{x_i\}_{i\in \mathbb{N}}$. Its elements will be called {\it words}. 
If $w\in F_\infty$, the length of $w$ will be denoted by $\ell(w)$.  A word in $n$ letters is an element of $F_\infty$ that requires at most $n$ distinct letters to be written in reduced form.  

\begin{definition}
  Let $G$ be a group. Let  $W\subseteq F_{\infty}$ be a nonempty set of words. The verbal subgroup $W(G)$ is defined as the subgroup of $G$ generated by the evaluation of all the elements of $W$ by the elements of $G$. 
   \end{definition}
   
Verbal subgroups are normal, in fact they are fully invariant, see for instance \cite[Lemma 2.3]{MR4502610}.

 \begin{definition}
Let $A$ and $B$ be groups. $[A,B]$ is the subgroup of $A\ast B$ generated by the elements of the form $[a,b]$ with $a\in A$ and $b\in B$.
\end{definition}
The group $[A,B]$ is free, and normal in $A\ast B$, see for instance  \cite[Chapitre I, $\S$ 1.3, Proposition 4]{serreTrees}.

\begin{definition} \label{def of verbal product}
Let $A,B$ be groups and let $W\subseteq F_{\infty}$ be a set of words. The verbal product between $A$ and $B$ (relative to the words in $W$), is the group 
\[A\vW B := \faktor{A\ast B}{W(A\ast B)\cap [A,B]}\]
\end{definition}

We list verbal subgroups and verbal products needed in this article, (see \cite[$\S$2]{MR4502610}).
\begin{examples}[of verbal subgroups] \label{examples of verbal groups} Given a group $G$, the verbal subgroup given by
\begin{enumerate}[label=(\roman*)]
\item the empty word is the identity of $G$;
\item \label{item1examples} the word $n_1:=[x_2,x_1]$  is  the commutator subgroup of $G$;
\item \label{item2examples} the words  $n_k:=[x_{k+1},n_{k-1}]$ with $k\in \mathbb N_{\geq 2}$ recursively yield the lower central series of $G$;
\item \label{item3examples} the words
\begin{align*}
s_1(x_1,x_2)&:=[x_1,x_2]\\
  s_k(x_1,\ldots,x_{2^k})
& :=
[s_{k-1}(x_{1},\ldots,x_{2^{k-1}}),s_{k-1}(x_{2^{k-1}+1},\ldots,x_{2^{k}})]
\end{align*}
recursively yield the derived series of $G$.
\end{enumerate}
\end{examples}

\begin{examples} [of verbal products]\label{examples of products} The words in the Examples \ref{examples of verbal groups} give the following verbal products.
\begin{enumerate}[label=(\roman*)]
   \item If $W=\{\text{The empty word}\}$,  $A\vW B = A\ast B$;
    \item \label{item1examples product}if $W=\{n_1\}$, with $n_1$ defined in \ref{examples of verbal groups}\ref{item1examples},then $A \vW B = A \times B$;
    \item  \label{item2examples product} if $W=\{n_k\}$, with $n_k$ defined in \ref{examples of verbal groups}\ref{item2examples}, then $A \vW B$  
    is called the $k$-nilpotent product; 
    \item  \label{item3examples product}if $W=\{s_k\}$, with $s_k$ as in \ref{examples of verbal groups}\ref{item3examples}, then $A\vW B$
    is called the $k$-solvable product. When $k=2$ it is also called the metabelian product.
\end{enumerate}
\end{examples}

\begin{proposition}\cite[Proposition 3.1]{MR4502610} \label{preservenilp} Let $A,B$ be finitely generated nilpotent groups. Let $W\subseteq F_\infty$ be the set of words in Example \ref{examples of verbal groups}\ref{item2examples}
 that defines a nilpotent product of groups.  Then $A \vW B$ is a finitely generated nilpotent group.
\end{proposition}

\section{The word problem for nilpotent and metabelian products}

Let $F_n$ be the free group in the $n$ letters $\{x_1,\ldots,x_n\}$ and let $G$ be the group $G=\langle x_1,\dots,x_n|R\rangle $, where $R$ denotes the set of relations defining $G$. By definition, the group $G$ has decidable word problem if there exists an algorithm with input $w\in F_n$ which returns the value $1$ if $w=1$ in $G$ and $0$ if not. 
If $H$ is a subgroup of $G$, the group $G$ has decidable membership problem in $H$ if there exists an algorithm with input $w\in F_n$ which returns the value $1$ if $w\in H$ and $0$ if not.

 \begin{theorem} \label{nilpotentcase}
 Let $A$ and $B$ be finitely generated groups with decidable word problem. Let $k\in\N$ and  let $W=\{n_k\}$ be the word in Example \ref{examples of verbal groups}\ref{item2examples},
 that defines the $k$-nilpotent product of groups, or let $W=\{s_2\}$ be the  word in Example \ref{examples of verbal groups}\ref{item3examples}
 that defines the 2-solvable product of groups. In all these cases, $A \vW B$ has decidable word problem.
 \end{theorem}
 
 \begin{proof}
 This can be deduced from structural results in \cite[Section $\S 2$]{MR4502610}. We provide a self contained proof instead.
 
By Definition \ref{def of verbal product},
given a word $x$ in the generators of $A\ast B$, we have $x=1$ in $A \vW B$ if and only if both $x\in[A,B]$ and $x\in W(A\ast B)$. We decide each condition separately.
For the first condition, let $\pi: A\ast B\to A\times B$ be the canonical projection, its kernel is $[A,B]$. The element $\pi(x)=(a, b)$ is computable from $x$ (collect the letters of $x$ belonging to $A$ and to $B$ respectively, preserving their order), and $x\in[A,B]$ if and only if $a=1$ in $A$ and $b=1$ in $B$; both conditions are decidable because $A$ and $B$ have decidable word problem.
For the second condition, recall  that in the nilpotent case, $W=\{n_k\}$, $W(A\ast B)$ is the $(k+1)$-st term of the lower central series of $A\ast B$, so that $\faktor{(A\ast B)}{W(A\ast B)}$ is a finitely generated nilpotent group of class at most $k$; such a group has decidable word problem,  (see for example \cite[Theorem 1.2]{MR49889}). Since $x\in W(A\ast B)$ if and only if $x=1$ in this quotient, the second condition is decidable as well.  Meanwhile when  $W=\{s_2\}$, $\faktor{(A\ast B)}{W(A\ast B)}$ is a finitely generated metabelian group and hence it has decidable word problem, (see for example  \cite[Theorem 2.1]{MR1202419} where they attribute it to Hall).
 \end{proof}
 
 \section{The word problem for solvable products}
 
 The proof of Theorem \ref{nilpotentcase} already 
hints that a similar result for $k$-solvable products might be false when $k\geq 3$. Indeed, the group $\faktor{(A\ast B)}{W(A\ast B)}$
is solvable when  $W=\{s_k\}$ is the word in Example \ref{examples of verbal groups}\ref{item3examples}.  A celebrated theorem of Kharlampovich, \cite{MR631441}, shows the existence of a finitely presented, solvable group of derived length $3$ with undecidable word problem. 

The rest of the article aims to exploit these observations to give families of examples of groups that show that for every $k\geq 3$, the $k$-solvable product of two groups with decidable word problem need not have decidable word problem.

By Definition \ref{def of verbal product} with $W=\{s_k\}$, and $k\geq 3$, it will be enough to exhibit two groups $A$ and $B$ such that the 
membership problem in $(A\ast B)^{(k)}\cap [A,B]$ is undecidable. For this, the most natural strategy will be to exhibit two groups $A$ and $B$ with $a\in A, b\in B$,  
such that the membership problem $[a,b]\in (A\ast B)^{(k)}$ is undecidable.  It also makes sense to take an ``easy'' group $B$, say $B=\Z$, and transfer the burden of the obstruction to the group $A$. This transfer is indeed achievable by means of the next two lemmas, which are interesting in their own right, since they are about localising (or spotting)  mixed commutators on free products.

 \begin{lemma}\label{in commutator}
Let $A,B$ be  groups. Let $M$ be a normal subgroup of $A$ and let $N$ be a normal subgroup of $B$.
If
$$
\varphi_{M,N}:A\ast B\longrightarrow(A/M)\times (B/N)
$$
is the canonical map and
$H_{M,N}:=\ker\varphi_{M,N}$,
then for every $a\in M$ and every $b\notin N$, 
$[a,b]\in H_{M,N}'$ if and only if  $a\in M'$.
\end{lemma}

\begin{proof}
The proof uses Kurosh Subgroup Theorem as it is presented in \cite{serreTrees} applied to $H_{M,N}$ as a subgroup of  $A\ast B$.\\
Observe that $b\notin H_{M,N}A$ because every element of $H_{M,N}A$ 
has trivial second coordinate under $\varphi_{M,N}$, while $\varphi_{M,N}(b)$ has non trivial second coordinate.

  Let $T$ be the Bass-Serre tree of $G=A*B$. Its vertices are
$G/A\sqcup G/B$; for each $g\in G$, the edge $g$ joins the
vertices $gA$ and $gB$. $G$ acts on $T$ with trivial edge
stabilizers and without inversions hence the same holds for the action of the group $H_{M,N}$.
Since $b\in B$, the following is a two-edge path in $T$:
$$
A\ \longleftrightarrow\ B\ \longleftrightarrow\ bA.
$$
Its two $A$-type endpoints represent different $H_{M,N}$-orbits, because
$b\notin H_{M,N}A$. Thus its image in the quotient graph $H_{M,N}\backslash T$
is an embedded path. Extend that image to a spanning tree of
$H_{M,N}\backslash T$, and by \cite[Chapitre I, $\S$ 3, Proposition 14]{serreTrees}, choose a lift of this spanning tree to a subtree of $T$
containing the  path $A\ \longleftrightarrow\ B\ \longleftrightarrow\ bA$.

The graph-of-groups description of $H_{M,N}$ associated with this lift
is a Kurosh decomposition, and since the edge stabilizers are trivial,
$H_{M,N}$ is the free product of the stabilizers of the lifted vertices,
together with a free group  (see \cite[Chapitre I, $\S$ 5.5, Th\'eor\`eme14]{serreTrees}). The two vertices $A$ and $bA$ belong
to the chosen lift, and their stabilizers are precisely
$$
Stab_{H_{M,N}}(A)=H_{M,N}\cap A=M,
\text{ and }
Stab_{H_{M,N}}(bA)=H_{M,N}\cap bAb^{-1}=bMb^{-1}.
$$
So these two subgroups occur simultaneously
as distinct free factors of $H_{M,N}$. In particular, if we denote $M^b=bMb^{-1}$, then 
$
H_{M,N}=M*M^{b}*K
$
for some group $K$ (which is itself a free product of the remaining
vertex stabilizers and a free group).

Then by taking abelianizations, it follows that
$$
M/M'\ \oplus\ M^b/(M^b)'\ \leq\ H_{M,N}/H_{M,N}'.
$$
As $[a,b]\in H_{M,N}'$ if and only its class in the abelianization of $H_{M,N}$ is equal to 0, computing the class of $[a,b]$ in the abelianization, we have 
\[\overline{[a,b]} = (\overline{a}, \overline{ba^{-1}b^{-1}}) \in M/M'\ \oplus\ M^b/(M^b)'\] 
and this is equal to zero if and only if  $\overline a=0$ in $M/M'$, if and only if $a\in M'$.
\end{proof}

\begin{lemma}\label{commutator localizations}
Let $A$ and $B$ be groups.
Then for every $k\geq2$,  every $a\in A$ and every $b\notin B^{(k-1)}$,
$$
[a,b]\in (A\ast B)^{(k)}
\quad\Longleftrightarrow\quad
a\in A^{(k)}.
$$
\end{lemma}

\begin{proof}
If $a\in A^{(k)}$, then $a\in (A\ast B)^{(k)}$.  Since $(A\ast B)^{(k)}$ is normal in $A\ast B$,
$[a,b]\in (A\ast B)^{(k)}$.

Conversely, suppose that $[a,b]\in (A\ast B)^{(k)}$.
We prove inductively that for all $0\leq r \leq k, a\in A^{(r)}$. The case $r=0$ is trivial.
Assume that, for some $0\leq r\leq k-1$,
$a\in A^{(r)}$. 
Apply Lemma \ref{in commutator} with
$
M=A^{(r)}$ and 
 $N=B^{(k-1)}$.
Let
$$
H_r
=
\ker\left(
A\ast B
\longrightarrow
(A/A^{(r)})\times(B/B^{(k-1)})
\right).
$$

Now $A/A^{(r)}$ has derived length at most $r$, while
$B/B^{(k-1)}$ has derived length at most $k-1$. Hence
$$
\bigl((A/A^{(r)})\times(B/B^{(k-1)})\bigr)^{(k-1)}=1.
$$
Therefore
$(A\ast B)^{(k-1)}\leq H_r$,
and
$(A\ast B)^{(k)}\leq H_r'$.
Since, by hypothesis,
$[a,b]\in (A\ast B)^{(k)}$, we have $[a,b]\in H_r'$.
Moreover, by the inductive hypothesis
$a\in A^{(r)}=M$,
and by the hypothesis of the Lemma,
$ b\notin B^{(k-1)}=N$.
Thus Lemma \ref{in commutator} implies
$a\in M' = (A^{(r)})' = A^{(r+1)}$. By induction, $a\in A^{(k)}$.
\end{proof}

\begin{theorem}\label{main theorem} Let $A$ be a group with undecidable membership problem in its $k$-derived subgroup $A^{(k)}$, $k\geq 2$, and let $B$ be a non-perfect group.  
Let $W=\{s_k\}$ be the word in Example \ref{examples of verbal groups}\ref{item3examples}
 that defines the $k$-solvable product of groups. Then $A \vW B$ has undecidable word problem.
\end{theorem}
 
 \begin{proof}
 Consider $a\in A$. Since $B$ is non perfect, there exists $b\in B\setminus B^{(k-1)}$.  
 By Lemma \ref{commutator localizations}, $[a,b]\in  (A\ast B)^{(k)}$ if and only if $a\in A^{(k)}.$ Since there is no algorithm to decide whether an element in $A$ belongs to its $k$-derived subgroup, there is no algorithm to decide whether $[a,b]\in  (A\ast B)^{(k)}\cap[A,B]$. Hence there is no algorithm to decide whether $[a,b]=1$ in  $A\vW B$.
 \end{proof}
 \begin{remark} This theorem is another instance that illustrates that non-perfect groups are the right building blocks for verbal products. Compare with \cite[Corollary 2.18]{SAS18}.
\end{remark}

Theorem \ref{main theorem} would be void for our purposes unless we exhibit finitely presented groups with decidable word problem and undecidable membership problem in their derived subgroups. The goal of next section is to construct such groups.
 
\section{Construction of hyperbolic groups with undecidable membership problem}

 We will use small cancellation theory. Our treatment is mainly borrowed from  \cite[Chapter 4, $\S$12]{MR1191619}, see also \cite{MR1812024} and \cite{MR1086661}.
 
Let $G=\langle x_1,\ldots, x_n | R\rangle$ be a finitely presented group. We assume that $R$ is cyclically reduced, namely that each $r\in R$ is reduced and that the first and last letters of $r$ are not inverses of each other. We define a larger set of relators $R^*$ by first adding to $R$ the set $R^{-1}$ of inverses of the words in $R\subseteq F_n$ and then by adding, for each $r=XY\in R\cup R^{-1}$, the word $YX$. The set $R^*$ is called the {\it symmetrization} of $R$. It is clear that $G=\langle x_1,\ldots, x_n | R^*\rangle$.
If $r_1=XY_1$ and $r_2=XY_2$ are distinct words in $R^*$ that begin with $X$, we say that $X$ is a {\it piece} relative to $R^*$. 

Let $\lambda \in \R$, $0< \lambda\leq 1$. We say that $G$ with the presentation $\langle x_1,\ldots, x_n | R^*\rangle$ satisfies the condition 
 $C'(\lambda)$ if for each $r=XY\in R^*$ with $X$ a piece, implies that $\ell(X)<\lambda \ell(r)$. 
If  $G$ admits a finite presentation satisfying the condition $C'(1/6)$, then $G$ has decidable word problem, (see for instance \cite[Chapter 4 $\S$12.3]{MR1191619}) and is hyperbolic, (see for instance  \cite[Theorem 36]{MR1086661}).

The following elementary observation will be our main tool to bound the length of pieces.
For a cyclically reduced word  $w=z_1\cdots z_m$ and $1\leq s\leq m$, we call $z_s\cdots z_mz_1\cdots z_{s-1}$
the cyclic permutation of $w$ \emph{starting at position $s$}; thus $R^*$ is the set of
cyclic permutations of the words in $R\cup R^{-1}$.
For instance, if $R=\{abcabd\}$, the cyclic permutations of $abcabd$ starting at positions
$1$ and $4$, namely $abcabd$ and $abdabc$, begin with $ab$, so $ab$ is a piece. On the other
hand, $bcabda$, which is the cyclic permutation that starts at position $2$, is the only cyclic permutation of $abcabd$ or
of $(abcabd)^{-1}=d^{-1}b^{-1}a^{-1}c^{-1}b^{-1}a^{-1}$ beginning with $bc$, and indeed no
piece relative to $R^{*}$ contains $bc$. This is a general fact:

\begin{lemma}\label{unique pairs}
Let $R\subseteq F_n$ be any set of cyclically reduced words, and let $u,v$ be letters such
that there is exactly one pair $(w,s)$, with $w\in R\cup R^{-1}$ and $1\leq s\leq \ell(w)$,
such that the cyclic permutation of $w$ starting at position $s$ begins with $uv$. Then no
piece relative to $R^*$ contains $uv$ as a subword.
\end{lemma}

\begin{proof}
Suppose that a piece $X=X_1uvX_2$ is a common prefix of two distinct words
$\rho_1,\rho_2\in R^*$. For $t=1,2$, $\rho_t$ is a cyclic permutation of
 some word $w_t\in R\cup R^{-1}$ starting at position $s_t$. 
 The letters of $\rho_t$ in positions
$\ell(X_1)+1$ and $\ell(X_1)+2$ are $u$ and $v$, so the cyclic permutation of $w_t$ starting
at position $s_t'$ begins with $uv$, where $1\leq s_t'\leq\ell(w_t)$ and
$s_t'\equiv s_t+\ell(X_1)$ modulo $\ell(w_t)$. By hypothesis it follows that $w_1=w_2$ and $s_1'=s_2'$, hence
$s_1=s_2$. Thus $\rho_1$ and $\rho_2$ are the cyclic permutation of the same word starting at
the same position, so $\rho_1=\rho_2$, a contradiction.
\end{proof}

In \cite{MR631441}, Kharlampovich constructed  a  finitely presented, solvable  group of derived length $3$ that has undecidable word problem. This group will be the departing point of the groups we construct in the next statement.

\begin{proposition} \label{prop princ}
Let $n\in \N$, $n\geq 3$ and let $\lambda \in \R$, $0< \lambda\leq 1$. There exists a finitely presented group $ K_n(\lambda)$ that satisfies the condition $C'(\lambda)$ such that there is no algorithm to determine whether its elements belong to the $n$-derived subgroup.
\end{proposition}

\begin{proof}
 Let \[K:=\langle x_1,\dots,x_d\mid r_1,\dots, r_l\rangle\] be Kharlampovich's group and let $M:= \max\{\ell(r_1),\dots,\ell(r_l)\}$. We may assume that each $r_i$ is a non-empty reduced word.
 The aim of the construction is to add generators and to ``enlarge'' the relations given in $K$ in order to have control over the length of the pieces. In order to do that, we will use a product of elements in the derived series in sufficiently many new letters.
 To be precise,
let $s_n$ be the word of Example \ref{examples of verbal groups}\ref{item3examples} which defines the $n$-subgroup of the derived series.
For a suitable $k\in \N$ to be determined later and for each $1\leq i\leq l$, consider the set of new letters $\{y_1^{i},\dots,y_{k2^n}^{i}\}$
and consider the word
\[S_i\left(y_1^{i},\dots,y_{k2^{n}}^{i}\right):= s_n(y_1^{i},\dots,y_{2^n}^{i}) s_n(y_{2^{n}+1}^{i},\dots,y_{2\,2^{n}}^{i})\dots s_n(y_{(k-1)2^{n}+1}^{i},\dots,y_{k2^{n}}^{i}).\]
Since   $\ell(s_n)=4^n$, it is clear that for each $1\leq i\leq l,\,\, \ell(S_i)=k4^n$.
 
\begin{claim} Let $k$ be such that $\frac{\max(M,4^{n-1})}{k4^n}< \lambda$. Then the group with the presentation
\[ K_n(\lambda):= \langle x_1,\dots,x_d, y_1^{i},\dots,y_{k2^{n}}^{i}, 1\leq i \leq l| r_1 S_1, \dots, r_l S_l\rangle\]
satisfies the condition $C'(\lambda)$.
\end{claim}
\begin{proof}[Proof of Claim]
Let $R:=\{r_1S_1,\dots,r_lS_l\}$. The idea is simple: there are so many different letters, and appearing in different words in $R$, 
that ``large'' pieces relative to $R^{*}$ can only come from subwords among the cyclic permutations of the 
words that define the group of Kharlampovich and their size is obviously bounded by $M$, or among 
subwords of a fixed $S_i$ and in this case, the largest pieces are the subwords of the form $s_{n-1}(\cdots)$. In what follows we detail the proof of all this.\\

We start by describing the first and last letters of some of the words
involved in the proof.
By induction on $t\geq1$, the word $s_t(z_1,\dots,z_{2^t})$ is cyclically reduced, begins with $z_1$
and ends with $z_{2^{t-1}+1}^{-1}$. Indeed, $s_1=z_1z_2z_1^{-1}z_2^{-1}$, and for $t\geq2$ we
have $s_t=pqp^{-1}q^{-1}$, where $p$ and $q$ are copies of $s_{t-1}$ written in the disjoint
sets of letters $\{z_1,\dots,z_{2^{t-1}}\}$ and $\{z_{2^{t-1}+1},\dots,z_{2^t}\}$; hence no
cancellation occurs, the first letter of $s_t$ is the first letter of $p$, and its last
letter is the inverse of the first letter of $q$.\\

Fix $i$; to lighten notation we write $y_t$ for $y^i_t$. For $1\leq j\leq k$ put
$b_j:=(j-1)2^n+1$ and $c_j:=b_j+2^{n-1}$, so that $S_i=f_1\cdots f_k$ with
$f_j:=s_n(y_{b_j},\dots,y_{b_j+2^n-1})$. By the previous paragraph, $f_j$ begins with
$y_{b_j}$ and ends with $y_{c_j}^{-1}$; in particular $S_i$ is cyclically reduced, begins with $y_{b_1}$
and ends with $y_{c_k}^{-1}$. Write $r_i=\xi_1\cdots\xi_e$, with $e\geq1$ and
$\xi_1,\dots,\xi_e\in\{x_1,\dots,x_d\}^{\pm1}$. Then
$$
r_iS_i=\xi_1\cdots\xi_e\,y_{b_1}\cdots y_{c_k}^{-1}
\qquad\text{and}\qquad
(r_iS_i)^{-1}=y_{c_k}\cdots y_{b_1}^{-1}\,\xi_e^{-1}\cdots\xi_1^{-1},
$$
so $r_iS_i$ is cyclically reduced.

In each of the following three steps we list some carefully chosen pairs of consecutive letters of the cyclic permutations of the words
 $(r_iS_i)^{\pm1}$. The key feature is that each of these pairs contains a $y$-letter of supra-index $i$, so it can not occur
in any $(r_jS_j)^{\pm1}$ with $j\neq i$, and it occurs at the beginning of only one cyclic permutation of one of the two words
 $(r_iS_i)^{\pm1}$. By Lemma
\ref{unique pairs}, applied to $R$, no piece relative to $R^*$ contains any of these pairs.

\smallskip
\noindent(1) In the cyclic permutations of the two words $(r_iS_i)^{\pm1}$, a letter of $\{x_1,\dots,x_d\}^{\pm1}$ is
adjacent to a $y$-letter only at the four pairs
$$
(\xi_e,\,y_{b_1})\quad,\quad (y_{c_k}^{-1},\,\xi_1)\quad,\quad (y_{b_1}^{-1},\,\xi_e^{-1})\quad,\quad
(\xi_1^{-1},\,y_{c_k}).
$$
Moreover each of these pairs occurs at most once in each cyclic permutation.
Their $y$-letters
$y_{b_1},y_{c_k}^{-1},y_{b_1}^{-1},y_{c_k}$ are pairwise distinct, since $c_k\neq b_1$.
Hence  by Lemma 
\ref{unique pairs} these pairs can not appear as subwords in pieces relative to $R^{*}$, 
and this entails that every piece relative to $R^{*}$ has only one of the next two possibilities:
\begin{enumerate}
\item[(i)] only has letters $x$. The size of this kind of pieces is bounded by $M$,
\item[(ii)] only has letters $y$ appearing in a single $S_i$, for some $i$.
\end{enumerate}
We begin to address the second of the former two possibilities: 

\smallskip
\noindent(2) Two consecutive letters of $S_i^{\pm1}$ that belong to different $s_{n}$-factors occur
only at the junctions between consecutive factors, namely the pairs
$$
(y_{c_j}^{-1},\,y_{b_{j+1}})\ \text{ in } S_i
\qquad\text{and}\qquad
(y_{b_{j+1}}^{-1},\,y_{c_j})\ \text{ in } S_i^{-1}=f_k^{-1}\cdots f_1^{-1},
\qquad 1\leq j<k;
$$
there is no junction between $f_k$ and $f_1$, since they are separated by $r_i^{\pm1}$. These
$2(k-1)$ pairs are pairwise distinct, because their first letters are.
 Moreover each of these pairs occurs at most once in each cyclic permutation of $S_i^{\pm1}$.
Hence  by Lemma
\ref{unique pairs} every piece relative to $R^{*}$ 
contained in $S_i^{\pm1}$ is a subword of a single factor $f_j^{\pm1}$.

\smallskip
\noindent(3) Fix $j$ and write $f_j=pqp^{-1}q^{-1}$, where $p$ and $q$ are copies of
$s_{n-1}$ in the letters $y_{b_j},\dots,y_{c_j-1}$ and $y_{c_j},\dots,y_{b_j+2^n-1}$,
respectively. Since $n\geq2$, by the first paragraph $p$ begins with $\alpha:=y_{b_j}$ and
ends with $\beta^{-1}$, where $\beta:=y_{b_j+2^{n-2}}$, while $q$ begins with
$\gamma:=y_{c_j}$ and ends with $\delta^{-1}$, where $\delta:=y_{c_j+2^{n-2}}$. The letters
$\alpha,\beta,\gamma,\delta$ are pairwise distinct. Two consecutive letters of
$f_j^{\pm1}$ from different halves of its alphabet occur only at the junctions between
blocks, namely the pairs
$$
(\beta^{-1},\gamma),\ (\delta^{-1},\beta),\ (\alpha^{-1},\delta)\ \text{ in } f_j,
\qquad
(\delta^{-1},\alpha),\ (\beta^{-1},\delta),\ (\gamma^{-1},\beta)\ \text{ in }
f_j^{-1}=qpq^{-1}p^{-1}.
$$
These six pairs are pairwise distinct. Hence by Lemma
\ref{unique pairs} every piece relative to $R^{*}$ 
contained in  a single factor $f_j^{\pm1}$ is a
subword of one of the blocks $p^{\pm1}$, $q^{\pm1}$. 

It is easy to see that $p$ (and $p^{-1}, q^{\pm1}$) are pieces relative to $R^*$. 
Indeed, on the one hand, since $r_iS_i= r_if_1\cdots f_k$, a suitable cyclic permutation of it gives the word 
$f_j\cdots f_{k}r_{i} f_{1} \cdots f_{j-1}$ that has $p$ as a prefix. On the other hand, 
since $(r_iS_i)^{-1}=f_k^{-1}\cdots f_1^{-1}r_i^{-1}$, a suitable cyclic permutation of it gives the word 
$f_{j}^{-1}\cdots f_1^{-1}r_i^{-1}f_k^{-1}\cdots f_{j+1}^{-1} =
qpq^{-1}p^{-1}\cdots f_1^{-1}r_i^{-1}f_k^{-1}\cdots f_{j+1}^{-1}$ and
 $pq^{-1}p^{-1}\cdots f_1^{-1}r_i^{-1}f_k^{-1}\cdots f_{j+1}^{-1}q$ is a cyclic permutation of it that has $p$ as a prefix.

\smallskip
Therefore every piece relative to $R^{*}$ has length at most $\max\{M,4^{n-1}\}$, and there are pieces of length $4^{n-1}$. Since
$\ell(r_iS_i)>\ell(S_i)=k4^{n}$ for all $1\leq i\leq l$, the claim follows.
\end{proof}
 
Let $W:=\{s_n\}$ be the word in Example \ref{examples of verbal groups}\ref{item3examples}. Observe that an element in $ K_n(\lambda)$ belongs to the $n$-derived subgroup $W( K_n(\lambda))$ if and only if its class in ${ K_n(\lambda)}/{W( K_n(\lambda))}$ is  $1$. Also observe that
by  definition, $S_i\in W( K_n(\lambda))$, for all $1\leq i\leq l$.
 
 \begin{claim}
The subgroup generated by the classes of $\{x_1,\ldots,x_d\}$ in the quotient ${ K_n(\lambda)}/{W( K_n(\lambda))}$ is isomorphic to $K$. \end{claim}
\begin{proof}[Proof of Claim]
Let $\phi:K\to K_n(\lambda)/W(K_n(\lambda))$ be defined on generators by sending $x_m$ to the
class of $x_m$, $1\leq m\leq d$. It is well defined: in $K_n(\lambda)$ we have
$r_i=S_i^{-1}\in W(K_n(\lambda))$, so each relator $r_i$ is mapped to $1$. Its image is the
subgroup generated by the classes of $x_1,\dots,x_d$, so it remains to show that $\phi$ is
injective.

Let $\psi:K_n(\lambda)\to K$ be defined by $\psi(x_m)=x_m$ and $\psi(y^i_t)=1$; it is well
defined since $\psi(r_iS_i)=r_i=1$ in $K$. Homomorphisms map verbal subgroups into verbal
subgroups, so
$$
\psi\big(W(K_n(\lambda))\big)=W(\psi(K_n(\lambda)))\subseteq W(K)=K^{(n)}\subseteq K^{(3)}=1,
$$
because $n\geq3$ and $K$ is solvable of derived length $3$. Hence $\psi$ induces
$\overline{\psi}:K_n(\lambda)/W(K_n(\lambda))\to K$, and $\overline{\psi}\circ\phi=\mathrm{id}_K$
since both fix each $x_m$. Therefore $\phi$ is injective.
\end{proof}

 Since $K$ does not have decidable word problem, it follows that there is no algorithm to decide whether an element in $ K_n(\lambda)$ that is written only in the letters $\{x_1,\ldots,x_d\}$ is in the $n$-derived subgroup $W( K_n(\lambda))$.
\end{proof}

By taking $\lambda =1/6$ we have the following Corollary:
\begin{corollary}
For any $n\geq 3$, there exists a hyperbolic group such that there is no algorithm to determine whether the elements of the group belong to the $n$-derived subgroup.
\end{corollary}

\begin{theorem}\label{last theorem}
For any $n\geq 3$, the $n$-solvable product of two finitely presented groups that have decidable word problem does not necessarily have decidable word problem.
\end{theorem}
\begin{proof}
Let $\lambda \in \R$, $0< \lambda\leq 1/6$  and let $K_{n}(\lambda)$  be the group of Proposition \ref{prop princ}. Let $B$ be any finitely presented non-perfect group with decidable word problem.\\
 Let  $W=\{s_n\}$ be the word in Example \ref{examples of verbal groups}\ref{item3examples},
 that defines the $n$-solvable product of groups. By Theorem \ref{main theorem}, $ K_{n}(\lambda)\vW B$ has undecidable word problem.
\end{proof}

\begin{examples}
For each $n\geq 3$, for each $\lambda\leq 1/6$, 
\begin{enumerate} 
    \item for each $r\in \N$ the $n$-solvable product $K_{n}(\lambda)\vW \Z^r$ has undecidable word problem;
    \item for each $j\in \N_{\geq2}$ the $n$-solvable product $K_{n}(\lambda)\vW C_j$ has undecidable word problem;
    \item  for each $m\in \N$ the $n$-solvable product $K_{n}(\lambda)\vW F_m$ has undecidable word problem;
\end{enumerate}
    but each factor has decidable word problem.
\end{examples}

\noindent
\textbf{Acknowledgments.}
The first author thanks Professor Gast\'on Garc\'ia for his encouragement to return to do research after sometime away from the scientific community. 
We thank the referee for a careful reading of the manuscript and for spotting a gap in our first proof of Lemma \ref{in commutator}.

\bibliographystyle{alpha}
\bibliography{bibliography}

\end{document}